\documentclass[11pt,reqno]{amsart}

\usepackage{amsfonts,amsmath,amsthm,mathscinet}
\usepackage{amssymb,epsfig}

\usepackage{cases, verbatim}
\usepackage[usenames]{color}
\usepackage{esint}

\usepackage{enumerate} 

\usepackage{color}
\definecolor{vert}{rgb}{0,0.6,0}

\usepackage{mathtools}
\usepackage{comment}

\usepackage{enumitem}

\usepackage[pagebackref,colorlinks,citecolor=blue,linkcolor=blue]{hyperref}

\newlist{thmlist}{enumerate}{1}
\setlist[thmlist]{label=(\roman{thmlisti}),noitemsep}

\usepackage[capitalize]{cleveref}

\Crefname{thm}{Theorem}{Theorems}
\Crefname{lem}{Lemma}{Lemmas}
\Crefname{listthm}{Theorem}{Theorems}
\Crefname{listlem}{Lemma}{Lemmas}

\makeatletter
\@namedef{subjclassname@2020}{
  \textup{2020} Mathematics Subject Classification}
\makeatother

\theoremstyle{plain}
\newtheorem{thm}[equation]{Theorem}

\theoremstyle{definition}
\newtheorem{defn}[equation]{Definition}
\newtheorem{cor}[equation]{Corollary}
\newtheorem{lem}[equation]{Lemma}
\newtheorem{prop}[equation]{Proposition}

\newtheorem{remark}[equation]{Remark}

\newtheorem{example}[equation]{Example}

\numberwithin{equation}{section}

\newcommand{\N}{\mathbb{N}}
\newcommand{\R}{\mathbb{R}}

\newcommand{\al}{\alpha}

\newcommand{\vep}{\varepsilon}

\newcommand{\supp}{{\rm supp}\,}

\newcommand{\loc}{{\rm loc}\,}

\newcommand{\osc}{\diam\,}
\newcommand{\dist}{{\rm dist}\,}

\DeclareMathOperator{\diam}{diam}
\DeclareMathOperator{\Lip}{Lip}

\begin{document}

\title[Absolutely continuous mappings
on doubling metric measure spaces]{Absolutely continuous mappings\\
	on doubling metric measure spaces
}

\author{Panu Lahti}
\address{Panu Lahti,  Academy of Mathematics and Systems Science, Chinese Academy of Sciences,
	Beijing 100190, PR China, {\tt panulahti@amss.ac.cn}}
\author{Xiaodan Zhou}
\address{Xiaodan Zhou, Analysis on Metric Spaces Unit, Okinawa Institute of Science and Technology Graduate University, 1919-1, Onna-son, 
Okinawa 904-0495, Japan, {\tt xiaodan.zhou@oist.jp}}


\date{\today}

\maketitle

\begin{abstract}
Following Mal\'y's \cite{Maly} definition of absolutely continuous functions of several variables,
we consider $Q$-absolutely continuous mappings $f\colon X\to V$ between a doubling
metric measure space $X$ and a
Banach space $V$.
The relation between these mappings and Sobolev mappings
$f\in N^{1,p}(X;V)$ for $p\ge Q$ is investigated. In particular, a locally $Q$-absolutely continuous
mapping  on an Ahlfors $Q$-regular space
is a continuous mapping in $N^{1,Q}_\loc(X;V)$, as well as differentiable almost everywhere in
terms of Cheeger derivatives provided $V$ satisfies the Radon-Nikodym property. Conversely, though a continuous Sobolev mapping
$f\in N^{1,Q}_\loc(X;V)$ is generally not locally $Q$-absolutely continuous, this implication holds
if $f$ is further assumed to be pseudomonotone. It follows that pseudomonotone mappings satisfying a relaxed quasiconformality condition are also
$Q$-absolutely continuous.

\end{abstract}

\section{Introduction}

Let $\Omega\subset \mathbb{R}^n$ be a domain. The definition of absolutely continuous functions on the real
line is generalized by Mal\'y \cite{Maly} to mappings $f\colon \Omega\to \mathbb{R}^d$: $f$ is called $n$-absolutely continuous if for every $\vep>0$ there exists $\delta>0$ such that for each finite
 family of disjoint balls $\{B_i\}$ in $\Omega$ we have
\[
\sum_i \mathcal{H}^n(B_i)<\delta\  \Rightarrow\  \sum_i (\diam(f(B_i)))^n<\vep.
\]

Properties of $n$-absolutely continuous mappings including weak differentiability, area and coarea
formulas, and the relation with Sobolev mappings $W^{1,p}(\Omega, \mathbb{R}^d)$ for $p\ge n$ are studied
in \cite{Maly}. We say that the mapping $f\colon \Omega\to \mathbb{R}^d$ is $n$-absolutely continuous in measure or
satisfies the Lusin property if $\mathcal{H}^n(f(E))=0$ whenever $\mathcal{H}^n(E)=0$. One can show that $f$
is absolutely continuous in measure provided $f$ is $n$-absolutely continuous.

This notion of absolutely continuous functions gives a unified approach to study the above mentioned
properties implied by conditions including monotonicity, finite dilatation and higher integrability of
the gradient. There are several works generalizing this definition in Euclidean spaces \cite{Hencl, Bon},
and extensions of this notion in metric spaces are considered by Marola and Ziemer \cite{MaZi}.
Using a rather different definition, absolutely continuous functions in a class of one-dimensional
metric spaces $X$ have been studied by the second author in \cite{Zhou}.

Consider a mapping $f\colon X\to Y$ between metric measure spaces $(X,d,\mu)$ and $(Y,d_Y,\nu)$. We can
define the class of absolutely continuous mappings by replacing the dimension $n$ with an appropriate dimension
of $X$ and replacing $\mathcal{H}^n$ with $\mu$. We say a measure $\mu$ is doubling if there exists $C>1$
such that $\mu(B_{2r}(x))\le C\mu(B_r(x))$ for all $x\in X$ and $r>0$.  For every doubling metric
space, we call the exponent  $Q\ge 1$ in the lower mass bound defined in \eqref{lmb}
the homogeneous dimension of $X$.
We define the class of $Q$-absolutely continuous mappings on $X$ and denote the collection as $AC^Q(X;Y)$.
A $Q$-absolutely continuous mapping $f\colon X\to Y$ is $Q$-absolutely continuous in measure, that is, $\mathcal{H}^Q(f(E))=0$ if $\mu(E)=0$. We can extend the definition of $Q$-absolutely continuous mappings
to a class of mappings with bounded $Q$-variation and call it $BV^Q(X;Y)$. One can verify that
$AC_{\loc}^Q(X;Y)\subset BV_{\loc}^Q(X;Y)$,
see Remark \ref{ACtoBV}.

Sobolev mappings $f\colon X\to Y$ between metric spaces are introduced by Heinonen, Koskela,
Shanmugalingam and Tyson \cite{HKST1}.
The Sobolev spaces are denoted by $N^{1,p}(\Omega;Y)$, and the corresponding Dirichlet spaces by
$D^p(\Omega;Y)$.
If the space $X$ is Ahlfors $Q$-regular, we obtain the following result. 

\begin{thm}\label{BVisSobolev}
Let $X$ be a complete Ahlfors $Q$-regular space with $Q>1$. If
$\Omega\subset X$ is open and bounded and
$f\in BV^Q(\Omega;Y)$, then
by choosing a suitable $\mu$-representative of $f$, we have $f\in D^{Q}(\Omega;Y)$.
\end{thm}

We will always understand $Y$ to be isometrically embedded into a Banach space $V$.
Cheeger \cite{Che99} shows that if $\mu$ is  doubling and admits a $ p$-Poincar\'e  inequality
for $p\ge1$, then the space $X$ acquires a measurable differentiable structure, and for every
$u\in N^{1,p}(X)$ one can define almost everywhere a gradient 
$x\mapsto Du(x) \in \R^N$, whose length is comparable to the minimal $p$-weak upper gradient of $u$. Thanks to
the Stepanov theorem proved by Wildrick and Z\"urcher \cite{WZ}, we can obtain that
mappings of bounded $Q$-variation
into $V$ are almost everywhere Cheeger differentiable,
provided that $V$ satisfies the Radon-Nikodym property.

\begin{prop}\label{aediff}
Let $X$ be an Ahlfors $Q$-regular space with $Q\ge 1$. Assume that there exists a measurable
differentiable structure $\{(U_\alpha, x_\alpha)\}$ for $(X,d,\mu)$,
and that $V$ satisfies the Radon-Nikodym property. If $\Omega\subset X$ is open and
$f\in BV^Q_\loc(\Omega;V)$, then $f$ is differentiable almost everywhere
with respect to the structure $\{(U_\alpha, x_\alpha)\}$.
\end{prop}

In the Euclidean case $\Omega\subset \mathbb{R}^n$,
every mapping $f\in N^{1,p}_\loc(\Omega;\R^n)$ for $p>n$ is immediately seen to be $n$-absolutely continuous
by a combination of the Sobolev embedding theorem that gives H\"older continuity,
and Young's inequality \cite[Theorem 4.1]{Maly}.
Assume the metric measure space $X$ with homogeneous dimension $Q\ge 1$ supports a $p$-Poincar\'e inequality,
with $p>Q$.
Due to
the dilation factor in balls on the right-hand side of the Sobolev embedding theorem,
the Euclidean argument does not apply to achieve $Q$-absolute continuity of $f\in N^{1,p}_\loc(\Omega;Y)$.
However, by additionally using the Hardy-Littlewood maximal function, the same result can still be achieved.

\begin{thm}\label{pgq}
Assume that $\mu$ is doubling, that
$p>Q$ where $Q\ge 1$ is the homogeneous dimension of $X$, and that $X$ supports a $p$-Poincar\'e
inequality.  Let $\Omega\subset X$ be a bounded open set.
Then every Sobolev mapping $f\in N^{1,p}(\Omega;Y)$, by choosing a suitable pointwise representative,
satisfies the Rado--Reichelderfer
condition \eqref{RR} locally in $\Omega$, and thus is locally $Q$-absolutely continuous,
as well as absolutely continuous in measure.
\end{thm} 

A continuous Sobolev mapping $f\in W^{1,n}$ does not necessarily satisfy the Lusin property
in the Euclidean space,  see for example \cite[Theorem 4.3]{HenKo}
for a continuous mapping $f\in W^{1,n}([-1,1]^n; [-1,1]^n)$ such that $f([-1,1]\times\{0\}^{n-1})=[-1,1]^n$.
More pathological examples can be found in \cite{HaZ} where a Sobolev homeomorphism
$f\colon S^n\to \mathbb{R}^{n+1}$ maps a set of Hausdorff dimension zero to Cantor set with
positive $n+1$-dimensional measure. With some additional conditions including being a
homeomorphism, or being continuous and open, or being a quasiregular mapping, one can obtain
the Lusin property of a Sobolev mapping $f\in W^{1,n}$ in Euclidean spaces, see for example \cite{MaMa, KoMaZu}
and the references therein. As pointed out in \cite[Page 458]{KoMaZu}, the essential qualitative
information leading to the validity of the Lusin property is the pseudomonotone condition, and a
continuous pseudomonotone Sobolev mapping $f\in W^{1, n}_\loc(\Omega)$ is shown to be locally
$n$-absolutely continuous in \cite[Theorem 4.3]{Maly}. A mapping $f\colon X\to Y$
is called pseudomonotone if there is $C_m\ge 1$ such that 
\[
\diam f(B_r(x))\le C_m\diam f(S_r(x))
\]
for all $x\in X$ and $r>0$,
where $S_r(x)$ denotes the sphere.
The following theorem recovers the above result in metric measure spaces. 

\begin{thm}\label{main}
Suppose $\mu$ is doubling with homogeneous dimension $Q>1$, and $X$
supports a $Q$-Poincar\'e inequality. Let $\Omega\subset X$ be an open set and let $f\in D^{Q}(\Omega;Y)$ be a continuous
pseudomonotone mapping. 
Then $f$ satisfies the condition \eqref{RR} with exponent $Q$ locally, and hence
is locally $Q$-absolutely continuous,
as well as absolutely continuous in measure.
\end{thm}

\begin{remark}
Similar results as Theorem \ref{pgq} and Theorem \ref{main} are obtained by Marola and Ziemer
in \cite[Proposition 5.1]{MaZi}. Compared to their result, our Theorem \ref{pgq} is stronger
in the sense that it implies the condition \eqref{RR} instead of weak \eqref{RR} in the case $p>Q$.

Another proof for every continuous pseudomonotone Sobolev mapping $f\in N^{1,Q}_\loc(X;V)$ being
absolutely continuous in measure
is given by Heinonen, Koskela, Shanmugalingam and Tyson \cite[Theorem 7.2]{HKST1}.
\end{remark}

%
%


In \cite{HKST1}, also the absolute continuity in measure of quasiconformal mappings in metric spaces with
appropriate structure assumptions is discussed.
In Euclidean spaces, it has been known since
Gehring \cite{Ge1,Ge2} that the definition of quasiconformality can be relaxed
to allow certain exceptional sets in the conditions on the distortion $H_f$ or $h_f$.
In a proper Ahlfors $Q$-regular spaces $X,Y$, with $X$ supporting a $1$-Poincar\'e inequality,
such a relaxed definition is shown to imply
Sobolev $N^{1,Q}_{\loc}(X;Y)$-regularity in \cite{BKR}.

In the recent work of the authors \cite{LaZh1}, the conditions on 
$h_f$ needed to obtain Sobolev $N^{1,Q}_{\loc}(X;Y)$-regularity are further weakened
by considering a weight function $w_Y$.
Combining this with Theorem \ref{main}, we obtain the following corollary.
Due to the weight function, the corollary gives something new already when $X=Y$ is the unweighted
Euclidean space.

\begin{cor}\label{Lusin}
Let $Q>1$. Suppose $(X,d,\mu)$ and $(Y,d_Y,\nu)$ are Ahlfors $Q$-regular,
and $X$ supports a $1$-Poincar\'e inequality. 
Suppose $w_Y\in L^1_{\loc}(Y)$ with $w_Y>0$ represented by \eqref{wY representative}.
Let $\Omega\subset X$ be a bounded open set, and
let $E\subset \Omega$ be a set of $\sigma$-finite $(Q-1)$-dimensional Hausdorff measure.
Assume $f\colon\Omega\to f(\Omega)\subset Y$ is a pseudomonotone homeomorphism such that
$f(\Omega)$ is open, $\int_{f(\Omega)}w_Y\,d\nu<\infty$,
$h_f(x)<\infty$ for every $x\in \Omega\setminus E$, and $\frac{{h_f(\cdot)^Q}}{w_Y(f(\cdot))}\in L^\infty(\Omega)$. 
Then $f$ is locally $Q$-absolutely continuous.
\end{cor}

The paper is organized as follows: we review definitions and some preliminary theorems in
Section \ref{prelim}, and in Section \ref{BVloc} we prove
Theorem \ref{BVisSobolev} and Proposition \ref{aediff}, which give
Sobolev regularity and almost everywhere differentiability of mappings of bounded $Q$-variation.
In Section \ref{Sobo} we prove Theorems \ref{pgq} and \ref{main}
and Corollary \ref{Lusin}, which give $Q$-absolute continuity of certain Sobolev and quasiconformal mappings.


\section{preliminaries}\label{prelim}

\subsection{Basic definitions}
Let $(X,d,\mu)$ and $(Y,d_Y, \nu)$ be metric measure spaces,
equipped with Borel regular outer measures $\mu,\nu$.
We assume that $Y$ is separable.

We say that the measure $\mu$ is doubling if there exists a constant $C_d\ge 1$
such that
\[
0<\mu(B_{2r}(x))\le C_d\mu(B_r(x))<\infty
\]
for all balls
$B_r(x)$ with $x\in X$ and $r>0$.
We understand balls to be open.
Given a ball $B=B_r(x)$, we sometimes denote $2B=B_{2r}(x)$;
in a metric space a ball (as a set) may not have a unique center point and radius, but we will understand
these to be prescribed whenever using this notation.
If $\mu$ is doubling, then there exist an exponent $Q\ge 1$ and a constant $C\ge 1$,
both depending only on the doubling constant, such that
\begin{equation}\label{lmb}
\frac{\mu(B_r(x))}{\mu(B_{r_0}(x_0))}\ge \frac{1}{C}\left(\frac{r}{r_0}\right)^Q
\end{equation}
for every ball $B_{r_0}(x_0)$ and every $x\in B_{r_0}(x_0)$, $r\le 2r_0.$ If such $Q$ and $C$ exist,
we say that the measure $\mu$ satisfies the lower mass bound \eqref{lmb}, and that $X$ has homogeneous
dimension $Q$.

We say that a space is proper if every closed and bounded set is compact. If $\mu$ is doubling,
then $X$ is separable, and if additionally $X$ is complete, then it is proper,
see e.g. \cite[p. 102--103]{HKST2}.

The $s$-dimensional Hausdorff content is denoted by $\mathcal H_R^s$, with $R>0$ and $s\ge 0$,
and then the $s$-dimensional Hausdorff measure $\mathcal H^s$ is obtained as the limit when $R\to 0$.
These definitions extend automatically to metric spaces.

We say that
the space $X$ is Ahlfors $Q$-regular if there exists a constant $C\ge 1$ such that
\[
C^{-1}r^Q\le\mu(B_r(x))\le Cr^Q
\]
whenever $x\in X$ and $0<r<{\rm diam}(X)$.
If $X$ is Ahlfors $Q$-regular, then the measure $\mu$
and the $Q$-dimensional Hausdorff measure $\mathcal{H}^Q$ are comparable.

 A continuous mapping $\gamma\colon [a,b]\to X$ is said to be a rectifiable curve if it has finite length. A rectifiable curve always admits an arc-length parametrization (see e.g. \cite[Theorem 3.2]{Haj}),
 so that we obtain a curve $\gamma\colon [0,\ell_{\gamma}]\to X$.
 Then if $g\colon \gamma([0,\ell_{\gamma}])\to [0, \infty]$ is a Borel function, we define
\[
\int_{\gamma} g\,ds:=\int_0^{\ell_{\gamma}}g(\gamma(s))\,ds.
\]

We always consider $1\le p<\infty$,
and by $\Omega\subset X$ we denote an open set.
As usual, a mapping $f\colon \Omega\to Y$ is said to 
be $\mu$-measurable if for every open set $W\subset Y$ we have that
$f^{-1}(W)$ is a $\mu$-measurable set.
By the Kuratowski embedding theorem,
every metric space, in particular $Y$, can be isometrically embedded into a Banach space,
for example $V=L^\infty(Y)$. See e.g. \cite[p. 100]{HKST2}.
We denote by $|\cdot|$ the norm on the real line as well as in the Banach space $V$.
By a simple mapping $f\colon \Omega\to V$ we mean a finite sum $f=\sum_i v_i\chi_{E_i}$, with $v_i\in V$
and the sets $E_i\subset \Omega$ are pairwise disjoint and $\mu$-measurable with $\mu(E_i)<\infty$.
The integral of a simple mapping is given by
\[
\int_{\Omega}f\,d\mu =\sum_{i}v_i\mu(E_i).
\]

\begin{defn}[Integrable mappings]
We say that a measurable mapping $f\colon \Omega\to V$ is (Bochner) integrable
if there exists a sequence of simple mappings
$\{f_j\}_{j=1}^{\infty}$ such that
\[
\int_\Omega |f_j-f|\,d\mu\to 0.
\]
Then the integral of $f$ is defined by
\[
\int_\Omega f\,d\mu=\lim_{j\to \infty}\int_\Omega f_j\,d\mu.
\]
Moreover, we say that $f\in L^p(\Omega;Y)$
if $|f|\in L^p(\Omega)$.
\end{defn}

If $f$ is $\mu$-measurable and $|f|\in L^1(\Omega)$, it follows that $f$ is integrable,
see \cite[Proposition 3.2.7]{HKST2}.

\begin{defn}[Upper gradient]\label{def:upper gradient}
Let $f\colon \Omega\to Y$. We say that a Borel
function $g\colon \Omega\to [0,\infty]$ is an upper gradient of $f$ in $\Omega$ if
\begin{equation}\label{eq:upper gradient inequality}
d_Y(f(\gamma(\ell_{\gamma})),f(\gamma(0)))\le \int_{\gamma}g\,ds
\end{equation}
for every curve $\gamma\colon [0,\ell_{\gamma}]\to \Omega$.
We use the conventions $\infty-\infty=\infty$ and
$(-\infty)-(-\infty)=-\infty$.
If $g\colon X\to [0,\infty]$ is a $\mu$-measurable function
and (\ref{eq:upper gradient inequality}) holds for $p$-almost every curve,
we say that $g$ is a $p$-weak upper gradient of $u$.
A property is said to hold for $p$-almost every curve
if there exists a nonnegative Borel function $\rho\in L^p(X)$ such that $\int_{\gamma}\rho\,ds=\infty$
for every curve $\gamma$ for which the property fails.
\end{defn}

The Sobolev class
$N^{1,p}(\Omega;Y)$ consists of those mappings $f\in L^p(\Omega;Y)$
for which there exists an upper gradient $g\in L^p(\Omega)$.
The Dirichlet class
$D^{p}(\Omega;Y)$ consists of those mappings
$f\colon \Omega\to Y$ for which there exists an upper gradient $g\in L^p(\Omega)$.
We also denote $N^{1,p}(X)=N^{1,p}(X,\overline{\R})$.

The integral average of $f\colon \Omega\to Y$ over a ball $B\subset\Omega$ is defined by
\[
f_B:=\fint_B f\,d\mu:=\frac{1}{\mu(B)}\int_B f\  d\mu
\]
if the integral exists;
note that $f_B$ might not be in $Y$, but it is in the Banach space $V\supset Y$.

\begin{defn}[Space supporting Poincar\'e inequality]\label{def:poincare}
Let $1\le p<\infty$. The space $X$ is said to support a $p$-Poincar\' e
inequality if there exist constants $C_P>0$ and $ \lambda\ge1$ such that the following holds for
every pair of functions
$u\colon X\to \overline{\mathbb{R}}$ and
$g\colon X\to [0,\infty]$, where $u\in L^1(X)$
and $g$ is an upper gradient of $u$:
\begin{equation}\label{poincare}
\fint_{B_r(x)} |u-u_{B_r(x)}|\,d\mu\le C_Pr\left(\fint_{B_{\lambda r}(x)}g^p\, d\mu\right)^{\frac{1}{p}}
\end{equation}
for every ball $B_r(x)$.
\end{defn}

The Poincar\'e inequality is defined by considering extended real-valued functions,
but we will be interested in mappings. If $X$ supports a $p$-Poincar\' e inequality,
by \cite[Theorem 8.1.49]{HKST2} we know that \eqref{poincare} holds also when $u$ is replaced by a Banach-space
valued-mapping $f\colon X\to V$.

By $C\ge 1$ we will often denote a constant depending only on the doubling, Ahlfors regularity,
and Poincar\'e inequality constants,
and the value of $C$ may change at each occurrence.

We recall the following embedding theorem on spheres for Sobolev mappings, from \cite[Theorem 7.1]{HaKo}.
It is only proved for real-valued functions, but the proof works also for
Banach-space valued mappings, since the Lebesgue differentiation theorem holds also for such mappings, see
\cite[p. 77]{HKST2}.

\begin{thm}\label{embedding on spheres}
	Suppose $\mu$ is doubling and that $p>Q-1$, where $Q\ge 1$ is the lower mass bound from \eqref{lmb}.
	Let $\Omega\subset X$ be open, and 
	suppose that a continuous mapping $f\colon\Omega\to V$ and a nonnegative function
	$g$ in $\Omega$ satisfy the $p$-Poincar\'e inequality \eqref{poincare}
	for every ball $B\subset \lambda B\subset \Omega$.
	Consider a ball $B_{r_0}(x_0)\subset 5\lambda B_{ r_0}(x_0)\subset \Omega$.
	Then for almost all $r\in (0,r_0)$,
	$f$ is H\"older continuous on the sphere $\{x\in X\colon d(x,x_0)=r\}$.  Moreover, there exists a
	constant $C$ depending only on the doubling constant and the constants in the Poincar\'e inequality,
	and $r\in (\frac{r_0}{2},r_0)$ such that
	\[
	d_Y(f(x), f(y))\le C d(x,y)^{1-\frac{Q-1}{p}}r_0^{\frac{Q-1}{p}}
	\Big(\fint_{5\lambda B_{r_0}(x_0)}g^p
	\, d\mu\Big)^{\frac{1}{p}}
	\]
	whenever $d(x,x_0)=d(y,x_0)=r$.
\end{thm}

\begin{defn}[Cheeger differentiable structure]\label{differentiability}
A measurable differentiable structure on $(X,d,\mu)$ is a finite or countable collection of pairs $\{U_\alpha,x_\alpha\}_{\alpha\in I}$ called coordinate patches with the following properties: 

\begin{enumerate}[label=\text{(\roman*)},font=\normalfont,leftmargin=*]
\item  Each $U_\alpha$, $\alpha\in I$,  is a  $\mu$-measurable subset of $X$  with positive measure, and
the complement of $\bigcup_{\alpha\in I}U_\alpha$ has zero $\mu$-measure.

\smallskip
\item  Each $x_\alpha\colon X\to \R^{N(\alpha)}$, $\alpha\in I$,  is a Lipschitz map on $X$, with $N(\alpha)\in \N$ bounded above independently of $\alpha\in I$.

\smallskip
\item For every Lipschitz function $f\colon X\to \R$ and every $\alpha\in I$  there exists an $L^\infty$-map  $Df^\alpha\colon X\to \R^{N(\alpha)}$ such that for $\mu$-a.e.\ $x\in U_\alpha$ we have 
\begin{equation}\label{diff}
\lim_{y\to x} \frac{1}{d(x,y)}\bigg|  f(y) -  f(x)-Df^\alpha(x)\cdot (x_\al(y) -x_\alpha(x)) \bigg|=0.
\end{equation}
\end{enumerate} 
\end{defn}
For those $x\in X$ for which a $Df^\alpha(x)$ exists so that \eqref{diff} holds, we say
that $f$ is differentiable at $x$. Cheeger proved the existence of a measurable differentiable structure,
 and thus a Rademacher differentiation theorem for
Lipschitz functions on metric spaces with a doubling measure and
supporting a Poincar\'e inequality. The Rademacher
theorem was later extended by Cheeger and Kleiner \cite{CK}
to mappings into a Banach space satisfying the Radon-Nikodym property;
the Banach space $V$ satisfies the Radon-Nikodym property if every Lipschitz function $f\colon \mathbb{R}\to V$ is differentiable almost everywhere with respect to Lebesgue measure.
We recall the differentiability of such mappings below. 

\begin{defn}
Let $\{U_\alpha,x_\alpha\}_{\alpha\in I}$ be a measurable differentiable structure on
$(X,d,\mu)$. Given a measurable subset $S$ of $X$, a mapping
$f\colon X\to V$ is differentiable almost everywhere in $S$ if there exists a collection of measurable
functions
$\{\partial f/ \partial x_n^\alpha\colon  U_\alpha \cap S\to V\}_{\alpha\in I, n\in \{1,2,\cdots, N(\alpha)\}}$
such that
\[
\lim_{y\to x}\frac{1}{d(x,y)}\bigg|f(y)-f(x)-\sum_{n=1}^{N(\alpha)}
((x_n^\alpha(y)-x_n^\alpha(x))\frac{\partial f}{\partial x_n^\alpha}(x))\bigg|=0
\]
for each $\alpha\in I$ and for almost every point in $S\cap U_\alpha$, and the collection $\{\partial f/ \partial x_n^\alpha\}$ is determined uniquely up to sets of measure zero in $S$.
\end{defn}

\subsection{Absolutely continuous mappings}

We always consider $Q\ge 1$. Also recall that $\Omega\subset X$ is always an open set.
We will use the notation
\[
\underset{B}{\osc}\,f=\sup\{d_Y(f(x),f(y))\colon x,y\in B\}.
\]

\begin{defn}
A $\mu$-measurable mapping $f\colon \Omega\to Y$ is $Q$-absolutely continuous if for every $\varepsilon>0$
there is $\delta>0$ such that for each finite family of disjoint balls $\{B_i\}$
in $\Omega$ satisfying $\sum_i\mu(B_i)<\delta$, we have
\[
\sum_i \left(\underset{B_i}{\osc}\,f\right)^Q<\vep.
\]
We denote the collection of $Q$-absolutely continuous mappings on $\Omega$ by $AC^Q(\Omega;Y)$.
\end{defn}

The above then clearly holds also for countable collections $\{B_i\}$.

We say that $f\colon \Omega\to Y$ is absolutely continuous in measure
if $\mathcal{H}^Q(f(N))=0$ whenever $\mu(N)=0$.
We have that a
$Q$-absolutely continuous mapping is absolutely continuous in measure. 
\begin{lem}\label{abm}
Suppose $\mu$ is doubling, with homogeneous dimension $Q\ge 1$.
If $f\in AC^Q(\Omega;Y)$, then $\mu(N)=0$ implies $\mathcal{H}^Q(f(N))=0$.
\end{lem}
\begin{proof}
Let $N\subset \Omega$ with $\mu(N)=0$. Since $f$ is $Q$-absolutely continuous in $\Omega$, there
exists $\delta>0$ such that for any countable collection of pairwise disjoint
balls $B_j \subset \Omega$ satisfying $\sum_j \mu(B_j)<\delta$, we have
\[
\sum_j (\underset{B_j}{\osc}\,f )^Q<\vep.
\]
Take an open set $W\subset \Omega$ containing $N$, with $\mu(W)<\delta/C_d^3$.
Then for every $x\in N$ there exists $r_x\in (0,1)$ such that $B_{5r_x}(x)\subset W$.
Consider the covering $\{B_{r_x}(x)\}_{x \in N}$.
By the $5$-covering theorem, see e.g. \cite[p. 60]{HKST2},
we can select an at most countable collection of balls
$B_{r_i}(x_i)\subset W$ such that $N\subset \bigcup_{i}B_{r_i}(x_i)$
and the balls $B_{r_i/5}(x_i)$ are disjoint. Thus also
\[
\sum_i \mu(B_{r_i}(x_i))\le \sum_i C_d^3\mu(B_{r_i/5}(x_i))\le C_d^3 \mu(W)<\delta.
\]
Let $\rho_i={\osc}\big(f(B_{r_i}(x_i)) \big)$. It follows that 
\[
f(B_{r_i}(x_i))\subset B_{\rho_i}(f(x_i)) \quad \text{and thus} 
\quad f(N)\subset\bigcup_{i} B_{\rho_i}(f(x_i)).
\]
Note that the numbers $\rho_i$ are necessarily uniformly bounded from above,
because otherwise $Q$-absolute continuity would be violated. 
Thus we can apply the $5$-covering theorem, and select an at most countable
collection of pairwise disjoint balls $\{B_{\rho_i}(f(x_i))\}_{i\in I}$ such that
 \[
f(N)\subset\bigcup_{i\in I} 5B_{\rho_i}(f(x_i)).
\]
It is clear that $B_{r_i}(x_i)\cap B_{r_j}(x_j)=\emptyset$ whenever $ B_{\rho_i}(f(x_i))\cap B_{\rho_j}(f(x_j))=\emptyset .$ Hence, we get a countable collection of pairwise disjoint balls
$\{B_{r_i}(x_i)\}_{i\in I}$ with
$\sum_{i\in I} \mu(B_{r_i}(x_i))<\delta.$ Denoting the constant appearing in the definition of
the Hausdorff measure by $C_Q$, we get
\[
\mathcal{H}_{\vep^{1/Q}}^Q(f(N))\le C_Q\sum_{i\in I} (5\rho_i)^Q
=5^Q C_Q \sum_{i\in I} (\underset{B_i}{\osc}\,f )^Q<5^Q C_Q \vep.
\]
Letting $\vep\to 0$, we get the result.
\end{proof}

\begin{remark}
	Note that when $X=\R^n$, we could deduce absolute continuity in measure from
	$Q$-absolute continuity by applying the Besicovitch covering theorem
	in the space $X$. In general metric spaces,
	for example in the first Heisenberg group equipped with the Kor\'anyi distance,
	we may not have such a covering theorem.
	Instead, we applied the $5$-covering theorem in the space $Y$.
\end{remark}

 Furthermore, if $Y$ is Ahlfors $Q$-regular, then a $Q$-absolutely continuous mapping $f$ satisfies the
 condition: if $\mu(N)=0$ then $\nu(f(N))=0.$

Besides $AC^Q$, we consider a more general class defined as follows. 
\begin{defn}\label{BVdef}
	Let $\Omega\subset X$ be an open set. Given a mapping $f\colon \Omega\to Y$,
	define
	\[
	V_Q(f,\Omega)=\sup \left\{\sum_{i} \left(\underset{B_i}{\osc}\,f\right)^Q\colon \{B_i\}
	\textrm{ is a finite family of disjoint balls in }\Omega\right\}.
	\]
	We define the seminorm
	\[
	\|f\|_{{\rm BV}^Q(\Omega;Y)}=\left(V_Q(f,\Omega)\right)^{1/Q}
	\]
	and say that $f$ is of bounded $Q$-variation if $\|f\|_{{\rm BV}^Q(\Omega;Y)}<\infty$,
	denoted by $f\in BV^Q(\Omega;Y)$. 
\end{defn}

We say that a mapping $f\colon \Omega\to Y$ satisfies a property locally if for every $x\in \Omega$,
there is $r>0$ such that the property holds in $B_r(x)\subset \Omega$.

\begin{lem}\label{locbdd}
Suppose $\mu$ is doubling and $X$ is connected.
Then every $f\in AC^Q_{\loc}(\Omega;Y)$ is continuous.
\end{lem}
\begin{proof}
	Suppose $f\in AC^Q_{\loc}(\Omega;Y)$ is not continuous at $x\in \Omega$.
	Then there exists $\vep>0$ such that for all $r>0$, we have
	$\underset{B_r(x)}{\osc}\,f>\vep$ for all $r>0$.
On the other hand, by the fact that
$X$ is connected, we have $\mu(B_r(x))\to 0$ as $r\to 0$, see \cite[Corollary 3.9]{BjBj11}.
This contradicts $f\in AC^Q_{\loc}(\Omega;Y)$.
\end{proof}

\begin{lem}\label{ACtoBV}
Suppose $\mu$ is doubling and $X$ is connected.
Then $AC_{\loc}^Q(\Omega;Y)\subset BV_{\loc}^Q(\Omega;Y)$.
\end{lem}
\begin{proof}
Let $f\in AC_{\loc}^Q(\Omega;Y)$ and let $x\in\Omega$.
By Lemma \ref{locbdd}, we find a ball $B_0\subset \Omega$ containing $x$
such that $\sup_{B_0}|f|<\infty$.
There exists
$\delta>0$ such that  $\sum_i (\diam f({B_i}))^Q<1$
if $\{B_i\}_i$ is a finite family of disjoint balls in $B_0$ with
$\sum_i\mu(B_i)<2\delta$. Let $n\in \mathbb{N}$
be such that $\mu(B_0)<n\delta.$
Consider an arbitrary finite collection of pairwise disjoint
balls $\{B_j\}$ in $B_0$.
Let $I_1$ consist of those indices $i$ for which $\mu(B_i)\ge \delta$, and $I_2$ consist of the
remaining indices.
Now $\sum_{i\in I_1} ({\osc}\,f({B_j}))^Q\le 2n\sup_{B_0}|f|<\infty$.
Divide the indices $I_2$ into at most $n$ subcollections
such that each subcollection satisfies $\sum_i\mu(B_i)<2\delta$.
Summing over all of these subcollections, we obtain
$\sum_{i\in I_2} ({\osc}\,f({B_j}))^Q<n$.
Thus $V_Q(f,B_0)\le 2n\sup_{B_0}|f|+n<\infty$.
\end{proof}

\begin{defn}\label{RRdef}
A mapping $f\colon \Omega\to Y$ is said to satisfy the Rado-Reichelderfer (RR) condition with exponent
$Q$ if there is a function $h\in L^1(\Omega)$ such that
\begin{equation}\label{RR}
\tag{RR}
\big(\underset{B_r(x)}{\osc}\,f\big)^Q\le \int_{B_r(x)} h\,d\mu
\end{equation}
for every ball $B_r(x)\subset \Omega$.
\end{defn}

A mapping satisfying the \eqref{RR} condition with exponent $Q$ is easily shown to
be $Q$-absolutely continuous, by using the absolute continuity of the integral. 

Given $x\in X$ and $r>0$, denote the sphere by
$S_r(x)=\{y\in X\colon d(y,x)=r\}$.
Note that in metric spaces, we always have $\partial B_r(x)\subset S_r(x)$,
where the inclusion can be strict.

\begin{defn}
We say that a mapping $f\colon \Omega \to Y$ is pseudomonotone if there exists a constant $C_m\ge 1$ such that
\[
\underset{B_r(x)}{\osc}\,f\le C_m\,\underset{S_r(x)}{\osc}\,f
\]
for each ball $B_r(x)\subset \overline{B_r(x)}\subset \Omega$.
\end{defn}

\emph{Throughout this paper we assume that $(X,d,\mu)$
	and $(Y,d_Y,\nu)$ are metric spaces equipped with Borel regular outer measures $\mu$ and $\nu$,
	and that $Y$ is separable. We always understand $Y$ to be embedded in a Banach space $V$.
}

\section{Sobolev regularity and almost everywhere differentiability of $BV^Q_\loc$}\label{BVloc}

\subsection{Sobolev regularity of $BV^Q_\loc$}

In this subsection, we always assume that $X$ is complete and $\mu$ is doubling, with
doubling constant $C_d\ge 1$.
Hence $X$ is separable and proper.
Moreover, here $\Omega\subset X$ is always a bounded open set.

We will show that a mapping $f\in BV^Q(\Omega;Y)$ belongs to the Dirichlet space $D^{Q}(\Omega;Y)$.
To achieve this, we construct locally Lipschitz mappings $f_\vep$
and their upper gradients $g_\vep$ such that $f_\vep\to f$ in $L_{\loc}^Q(\Omega;Y)$
and the $g_\vep$'s converge weakly to $g\in L^Q(\Omega)$,
from which it follows that $f\in D^Q(\Omega;Y)$.

We first recall a Lipschitz partition of unity of $\Omega$.
For any $\vep>0,$ consider a Whitney-type covering $\{B^{\vep}_i\}_i$
of $\Omega$, with $B^{\vep}_i=B_{r_{\vep,i}}(x_{\vep,i})$ where $r_{\vep,i}\le \vep$, and $4B^{\vep}_i\subset\Omega$.
If $4B^{\vep}_i\cap 4B^{\vep}_j\neq \emptyset$, then $r_{\vep,i}\le 2r_{\vep,j}$.
Moreover, there exists a constant $M=M(C_d)$ determined by the doubling constant
$C_d$ such that $1\le \sum_i\chi_{12B^{\vep}_i}(x)\le M$ for all $x\in \Omega$.
For such a construction, see e.g. \cite[Lemma 4.1.15]{HKST2}.

Define $\psi_i=\min\{\frac{1}{\vep}\dist(\cdot, X\setminus 2B_i^\vep), 1\}$ and
$\phi_i=\frac{\psi_i}{\sum_k \psi_k}$. We can verify that $\phi_i\colon X\to [0,1]$ satisfy

\begin{itemize}
\item[(i)] $\sum_i \phi_i(x)=1$ and the number of nonzero terms in this sum is bounded by the constant
$M$;
\item[(ii)] $\phi_i=0$ in $X\setminus 2B_i^\vep$; 
\item[(iii)] $\phi_i$ is $M/r_{\vep,i}$-Lipschitz continuous. 
\end{itemize}

Given $f\in L^1_{\loc}(\Omega;Y)$, we define the discrete convolution approximation
$f_\vep\colon \Omega\to V$ as
\[
f_\vep(x)=\sum_i f_{B_i^\vep}\phi_i(x),\quad x\in\Omega.
\]
Note that $f$ is assumed to take values in $Y$, but $f_{\vep}$ takes values in $V\supset Y$. 
The following basic properties hold for $f_\vep$:
\begin{itemize}
\item[(1)] $f_\vep\in \Lip_{\loc}(\Omega;V)$;
\item[(2)] $f_\vep\to f$ as $\vep\to 0$ uniformly on every compact subset of $\Omega$ if $f$ is continuous;
\item[(3)] if $f\in L^p(\Omega;Y)$, then $\|f_\vep\|_{L^p(\Omega;V)}\le C_0\|f\|_{L^p(\Omega;V)}$,
where $C_0\ge 1$ depends only on $C_d$.
\end{itemize}
The proofs for the above properties are similar to Lemma 8.1, Lemma 8.2 in \cite{HKST1}
and \cite[Lemma 5.3(2)]{HKT}.

\begin{lem}\label{lem:Lp convergence}
Let $f\in L^p_{\loc}(\Omega;V)$ for some $1\le p<\infty$.
Then $f_\vep\to f$ in $L_{\loc}^p(\Omega;V)$ as $\vep\to 0$.
\end{lem}

\begin{proof}
By using a cutoff function, we can assume that $f\in L^p(\Omega;V)$.
Fix $\vep>0$. By \cite[Proposition 3.3.52]{HKST2}, there exists 
$\hat{f}\in C_c(\Omega;V)$ with $\Vert \hat{f}-f\Vert_{L^p(\Omega;V)}<\vep/(3C_0)$.
We can assume that $\mu(\supp \hat{f})>0$.
Denote the discrete convolution of $\hat{f}$ at scale $\delta>0$ by $\hat{f}_\delta$.
Denote the $\delta$-neighborhood of a set $A$ by $N_\delta(A)$.
By property (2), we find $\delta>0$ sufficiently small so that
\[
\Vert \hat{f}_\delta-\hat{f}\Vert_{L^{\infty}(\Omega;V)}<\frac{\vep}{3\cdot \mu(N_{3\delta}(\supp \hat{f}))^{1/p}},
\]
and so
\[
\Vert \hat{f}_\delta-\hat{f}\Vert_{L^{p}(\Omega;V)}\le 
\Vert \hat{f}_\delta-\hat{f}\Vert_{L^{\infty}(\Omega;V)}\mu(N_{3\delta}(\supp \hat{f}))^{1/p}<\vep/3.
\]
Using also property (3), we get
\begin{align*}
\Vert f_\delta-f\Vert_{L^p(\Omega;V)}
&\le \Vert f_\delta-\hat{f}_\delta\Vert_{L^p(\Omega;V)}+\Vert \hat{f}_\delta-\hat{f}\Vert_{L^p(\Omega;V)}+\Vert \hat{f}-f\Vert_{L^p(\Omega;V)}\\
&\le \Vert (f-\hat{f})_\delta\Vert_{L^p(\Omega;V)}+\vep/3+\Vert \hat{f}-f\Vert_{L^p(\Omega;V)}\\
&\le C_0\Vert f-\hat{f}\Vert_{L^p(\Omega;V)}+\vep/3+\Vert \hat{f}-f\Vert_{L^p(\Omega;V)}\\
&< \vep.
\end{align*}
\end{proof}

Next, define a function $g_\vep$ by
\begin{equation}\label{ug}
g_\vep(x)=\sum_i\frac{\diam f(4B_i^\vep)}{r_{\vep,i}}\chi_{B_i^\vep}(x),\quad x\in \Omega.
\end{equation}

Moreover, let $L(f_\vep, x,r)=\sup\{|f_\vep(x)-f_\vep(y)|\colon  d(x,y)\le r\}$,
and define the local Lipschitz constant by
\[
\begin{aligned}
\Lip f_\vep(x)&=\limsup_{r\to 0}\frac{L(f_\vep, x,r)}{r}.
\end{aligned}
\]

We have the following lemma.
\begin{lem}\label{cg}
There exists a constant $C$ depending on $C_d$ such that
$\Lip f_\vep\le  Cg_\vep$. In particular, $Cg_\vep$ is an upper gradient of $f_\vep$ in $\Omega$.
\end{lem}
\begin{proof}
Let $x\in \Omega$. Choose a ball $B_j^\vep$ containing $x$.
Let $I$ denote the collection of indexes for which
$2B_i^\vep\cap B_j^{\vep}\neq \emptyset$ for $i\in I$.
Then ${\rm card}(I)\le M$. 
Now for every $y\in B_j^{\vep}$,
\[\sum_i (\phi_i(x)-\phi_i(y))f_{2B_i^\vep}=\sum_{i\in I} (\phi_i(x)-\phi_i(y))f_{2B_i^\vep},
\]
and $B_i^\vep\subset 4B_j^\vep$ if $i\in I.$

It follows that
\[
\begin{aligned}
|f_\vep(x)-f_\vep(y)|
&=|f_\vep(x)-f_\vep(y)-f_{B_j}\sum_i(\phi_i(x)-\phi_i(y))|\\
&=|\sum_{i\in I} (\phi_i(x)-\phi_i(y))(f_{B_i}-f_{B_j})|\\
&\le \sum_{i\in I} |\phi_i(x)-\phi_i(y)||f_{B_i}-f_{B_j}|\\
&\le \frac{2M^2}{r_{\vep,j}}d(x,y)\diam f(4B_j^\vep),
\end{aligned}
\]
where the last inequality follows from the fact that $\phi_i$ is $M/r_{\vep,i}$-Lipschitz
and thus $2M/r_{\vep,j}$-Lipschitz,
and
${\rm card}(I)\le M$. Hence, we obtain
\[
\limsup_{r\to 0}\frac{L(f_\vep, x,r)}{r}\le \frac{2M^2}{r_{\vep,j}} \diam f(4B_j^\vep)\le Cg_\vep (x).
\]
\end{proof}

\begin{proof}[Proof of Theorem \ref{BVisSobolev}]
Let $f\in BV^Q(\Omega;Y)$.
From the definition of $BV^Q(\Omega;Y)$, we see that
$f\in L_{\loc}^{\infty}(\Omega;Y)\subset L_{\loc}^Q(\Omega;Y)$.
By Lemma \ref{lem:Lp convergence}, $f_\vep\to f$ in $L_{\loc}^Q(\Omega;V)$ as $\vep\to 0$.
Recall that $\sum_i \chi_{12B_i^\vep}(x)\le M$,
and that if $4B^{\vep}_i\cap 4B^{\vep}_j\neq \emptyset$, then $r_{\vep,i}\le 2r_{\vep,j}$.
Thus we can divide the collection of balls
$\{4B_i^\vep\}_{i}$ into at most $M+1$ collections of disjoint balls
$\{4B_{i}^\vep\}_{i\in I_1},\ldots,\{4B_{i}^\vep\}_{i\in I_{M+1}}$.
Otherwise, there would be a ball $4B_i^\vep$ whose center is in $M+1$ balls
$4B_i^\vep$.
Defining $g_\vep$ as in \eqref{ug},
we have by the triangle inequality, and then using Ahlfors $Q$-regularity,
\[
\begin{aligned}
&\left(\int_{\Omega} g^Q_\vep\,d\mu\right)^{1/Q}
= \left(\int_{\Omega}\Bigg(\sum_{i} \frac{\diam f(4B_{i}^\vep)}{r_{\vep,i}}\chi_{2B_{i}^\vep}\Bigg)^Q\,d\mu\right)^{1/Q}\\
&\quad \le \left(\sum_{i\in I_1}\Bigg( \frac{\diam f(4B_{i}^\vep)}{r_{\vep,i}}\Bigg)^Q\mu(2B_{i}^\vep)\right)^{1/Q}
+\ldots+ \left(\sum_{i\in I_{M+1}}\Bigg( \frac{\diam f(4B_{i}^\vep)}{r_{\vep,i}}\Bigg)^Q\mu(2B_{i}^\vep)\right)^{1/Q}\\
&\quad \le  C\left(\sum_{i\in I_1}(\diam f(4B_{i}^\vep))^Q\right)^{1/Q}
+\ldots+ C\left(\sum_{i\in I_{M+1}}(\diam f(4B_{i}^\vep))^Q\right)^{1/Q}\\
&\quad \le (M+1) C V_Q(f,\Omega)<\infty.
\end{aligned}
\]
Then by reflexivity, there exists a $g\in L^Q(\Omega)$ such that $g_\vep\to g$ weakly in $L^Q(\Omega)$.
Combining this with the fact that $f_\vep\to f$ in $ L_{\loc}^Q(\Omega)$, and the fact that $Cg_\vep$ is an
upper gradient of $f_\vep$ in $\Omega$ by
Lemma \ref{cg}, we deduce that $Cg$ is a $Q$-weak upper gradient of $f$ in $\Omega$ by
choosing an appropriate $\mu$-representative of $f$ \cite[Theorem 7.3.8]{HKST2}.
Denoting this representative still by $f$, we have $f\in D^{Q}(\Omega;Y)$, and the proof is complete.
\end{proof}

\subsection{Almost everywhere differentiability of $BV^Q_\loc$}

We recall that a Banach space $V$ satisfies the Radon-Nikodym property if every Lipschitz function $f\colon \mathbb{R}\to V$ is differentiable almost everywhere with respect to Lebesgue measure. Wildrick and Z\"urcher \cite[Theorem 4.2]{WZ} proved the following Stepanov differentiability theorem. 
\begin{thm}\label{thm:Stepanov}
Let $(X,d,\mu)$ be a complete metric measure space with $\mu$ doubling. Let $V$ be a Banach space with the Radon-Nikodym property. Assume there is a measurable
differentiable structure $\{(U_\alpha, x_\alpha)\}$ on $(X,d,\mu)$. Then a measurable function
$f\colon X\to V$ is $\mu$-a.e. differentiable in $S(f)=\{x\in X\colon \Lip f(x)<\infty\}$
with respect to $\{(U_\alpha, x_\alpha)\}$.

\end{thm}
Hence, to complete the proof of almost everywhere Cheeger differentiability of mappings
$f\in BV^Q(\Omega;V)$ for $V$ with Radon-Nikodym property, it suffices to show that the local Lipschitz constant $\Lip f$
is finite $\mu$-almost
everywhere. The argument  in $\mathbb{R}^n$ by Mal\'y \cite[Theorem 3.3]{Maly}
can be extended to metric spaces as below.

\begin{proof}[{Proof of Proposition \ref{aediff}}]
As before, for $L(f, x,r)=\sup\{|f(x)-f(y)|\colon d(x,y)\le r\}$, we have
\[
\begin{aligned}
\Lip f(x)&=\limsup_{r\to 0}\frac{L(f, x,r)}{r}.
\end{aligned}
\]
We will show that $\int_\Omega \left(\Lip f\right)^Qd\mu \le CV_Q(f,\Omega).$
Note that by the definition of the Lebesgue integral, we have
\[
\int_\Omega \left(\Lip f\right)^Q\, d\mu=\sup\left\{\int_\Omega g\, d\mu\colon \text{$g$ is a simple function and } 0\le g\le (\Lip f)^Q\chi_\Omega\right\}.
\]
Simple functions are finite sums $\sum_i a_i\chi_{E_i}$, with $a_i\ge 0$,
and by an inner regularity property of the Borel regular outer measure $\mu$
we can assume that the sets $E_i$ are closed, see e.g. \cite[Proposition 3.3.37]{HKST2}.
We conclude that it suffices to consider upper semicontinuous functions $g$.
Assume $g$ is such a function and  let $E=\{g>0\}\subset \Omega$. For $x\in E$, denote
by $k(x)$ the integer such that $$2^{k(x)}\le g(x)<2^{k(x)+1}.$$ By the fact that $g$
is upper semicontinuous and $g(x)\le (\Lip f(x))^Q$,
for each $x\in E$ we can find a radius $r=r(x)$ small enough  such that $B_r(x)\subset\Omega$,
\[
B_{5r}(x)\subset \{g<2^{k(x)+1}\},
\]
and 
\[
2^{k(x)}r^Q\le g(x)r^Q< 2L(f,x,r)^Q<  2\diam f(B_r(x))^Q.
\]

By the $5$-covering theorem, see e.g. \cite[p. 60]{HKST2},
from the collection $\left\{ B_{r(x)}(x)\right\}_{x\in E}$
we can select an at most countable disjoint subcollection $\{B_i=B_{r_i} (x_i)\}_i$ such that
\[
E\subset \bigcup_i 5B_i.
\]
Let $I_k=\{i\colon k(x_i)\ge k\}$ and
\[
E_k=\{x\in E\colon 2^k\le g(x)<2^{k+1}\}.
\]
Since we chose $r(x_i)$ such that $B_{5r_i}(x_i)\subset \{g<2^{k(x_i)+1}\} $,
we know that $E_k\subset \underset{i\in I_k} \bigcup 5B_i$. Using also Ahlfors
$Q$-regularity, it follows that
\[
\mu(E_k)\le \sum_{i\in I_k} \mu(5B_i)\le C\sum_{i\in I_k} r_i^Q
\le C\sum_{i\in I_k} 2^{-k(x_i)}\diam f(B_{r_i}(x_i))^Q.
\]
Hence, we have
\[
\begin{aligned}
\int_\Omega g\, d\mu&=\sum_{k=-\infty}^{\infty} \int_{E_k} g\, d\mu\\
&\le\sum_{k=-\infty}^{\infty} 2^{k+1}\mu(E_k)\\
&\le C\sum_{k=-\infty}^{\infty}\sum_{i\in I_k} 2^{k+1-k(x_i)}\diam f(B_{r_i}(x_i))^Q\\
&= C\sum_{i}\sum_{k\le k(x_i)} 2^{k+1-k(x_i)}\diam f(B_{r_i}(x_i))^Q\\
&\le CV_Q(f,\Omega)<\infty.
\end{aligned}
\]
We deduce that $\Lip f\in L^Q(\Omega)$.
From Theorem \ref{thm:Stepanov}
we deduce that $f$ is differentiable almost everywhere with
respect to the given measurable differentiable structure. 
\end{proof}

\section{Absolute continuity of Sobolev mappings}\label{Sobo}

In the previous section we considered properties of the classes $AC^Q$ and $BV^Q$.
In this section we consider the converse question of when a mapping can be shown to be
an $AC^Q$-mapping.
We show this for Sobolev mappings $N^{1,p}$ in three cases:
when $p>Q$, when $p=Q$ and the mapping is pseudomonotone, and finally when
the Sobolev property of the mapping is obtained from a suitable quasiconformality condition.

We define the non-centered Hardy--Littlewood maximal function of a locally integrable nonnegative function
$g\in L^1_{\loc}(X)$ by
\[
\mathcal Mg(x):=\sup_{B\ni x}\,\fint_{B}g\,d\mu,\quad x\in X,
\]
where the supremum is taken over all balls containing $x$.

\subsection{Sobolev mapping $f\in N^{1,p}$ with $p>Q$ satisfies \eqref{RR}}

We show that a Sobolev mapping $f\in N^{1,p}(\Omega;Y)$ with $p>Q$  satisfies
the Rado-Reichelderfer condition \eqref{RR}
of Definition \ref{RRdef}, which implies $Q$-absolute continuity.

\begin{proof}[{Proof of Theorem \ref{pgq}}]
By Keith--Zhong, see \cite[Theorem 5.1]{BjBj} as well as \cite[Theorem 1.0.1]{KeZh}
for the original result in complete spaces,
$X$ supports a $q$-Poincar\'e inequality for some  $q\in (Q,p)$.
Let $f\in N^{1,p}(\Omega;V)\subset N^{1,q}(\Omega;V)$ and let
$g\in L^p(\Omega)$ be an
upper gradient of $f$ in $\Omega$. The Sobolev embedding theorem
(see e.g. \cite[Theorem 9.1.36(iii)]{HKST2})
states that $f$ is locally H\"older continuous on $\Omega$ by a suitable choice of representative,
which we still denote by $f$.
The Sobolev embedding theorem also states for some constant $\sigma\ge 1$, we have
\[
\underset{B_r(x)}{\diam}\,f
\le Cr\left(\fint_{ B_{\sigma r}(x)}g^q\, d\mu\right)^{\frac{1}{q}}
\]
for every ball $B_r(x)\subset B_{\sigma r}(x)\subset\Omega$.
Consider a fixed ball $B_0=B_{R_0}(x_0)$ with
$B_0\subset 3\sigma B_0\subset \Omega$.
Then consider a ball $B_{r}(x)\subset B_0$; note that $B_{\sigma r}(p)\subset \Omega$.
First note that by \eqref{lmb}, we have
\begin{equation}\label{lmb2}
\frac{\mu(B_r(x))}{\mu(B_0)}\ge \frac{1}{C}\left(\frac{r}{R_0}\right)^Q.
\end{equation}
Let $\rho$ be the zero extension of $g^q$ from $\Omega$ to the whole space.
Now for all $z\in B_r(x)$,
\[
\left(\underset{B_r(x)}{\diam}\,f\right)^q
\le Cr^q \fint_{ B_{\sigma r}(x)}g^q\, d\mu
\le  Cr^q \mathcal{M}\rho(z).
\]
Integrating both sides with respect to $z$ over $B_r(x)$, by \eqref{lmb2} we obtain
\[
\left(\underset{B_r(x)}{\osc}\,f\right)^q \le \frac{Cr^{q}}{\mu(B_r(x))} \int_{B_r(x)}\mathcal{M}\rho \, d\mu
\le Cr^{q-Q}\frac{R_0^Q}{\mu(B_0)}\int_{B_r(x)}\mathcal{M}\rho \, d\mu.
\]
Finally, taking power $\frac{Q}{q}$ on  both sides and applying Young's inequality, it follows that 
\[
\begin{aligned}
\left(\underset{B_r(x)}{\osc}\,f\right)^Q
& \le C \left(\frac{R_0^Q}{\mu(B_0)}\right)^{Q/q} r^{Q(1-\frac{Q}{q})}\left( \int_{B_r(x)}\mathcal{M}\rho \, d\mu\right)^{\frac{Q}{q}}\\
&\le C \left(\frac{R_0^Q}{\mu(B_0)}\right)^{Q/q} \left(r^Q+\int_{B_r(x)}\mathcal{M}\rho \, d\mu\right)\\
&\le C \left(\frac{R_0^Q}{\mu(B_0)}\right)^{Q/q} \int_{B_r(x)}
\Bigg[\frac{R_0^Q}{\mu(B_0)}+\mathcal{M}\rho \Bigg]\, d\mu
\end{aligned}
\]
again by \eqref{lmb2}.
Since $\rho \in L^{\frac{p}{q}}(X)$ with $\frac{p}{q}>1$, the Hardy-Littlewood
theorem implies
that $\mathcal{M} \rho\in L^{\frac{p}{q}}(X)\subset L^{\frac{p}{q}}(B_0)\subset L^1(B_0)$,
see e.g. \cite[Theorem 3.5.6]{HKST2}.
Let $h=1+\mathcal{M}\rho \in L^1(B_0)$.
Thus $f$ satisfies the condition \eqref{RR} locally with exponent $Q$.
As a result, $f$ is locally $Q$-absolutely continuous in $\Omega$.

Since $X$ is separable, we can cover $\Omega$ by countably many balls
$B_i\subset 3\lambda B_i\subset \Omega$,
and applying Lemma
\ref{abm} in each $B_i$, we obtain absolute continuity in measure on $\Omega$.
\end{proof}
\subsection{Continuous pseudomonotone Sobolev mapping $f\in N^{1,Q}$ satisfies \eqref{RR}}


\begin{proof}[{Proof of Theorem \ref{main}}]
	There is $p\in (Q-1,Q)$ such that $X$ supports a $p$-Poincar\'e inequality, see
	\cite[Theorem 5.1]{BjBj}.
Consider a fixed ball $B_0=B_{R_0}(x_0)$ with
$B_0\subset 21\sigma B_0\subset \Omega$, where $\sigma$ is the dilation factor
from the $p$-Poincar\'e inequality.
Then consider a ball $B_{r}(x)\subset B_0$.
First note that by \eqref{lmb}, we have
\begin{equation}\label{lmb3}
\frac{\mu(B_r(x))}{\mu(B_0)}\ge \frac{1}{C}\left(\frac{r}{r_0}\right)^Q.
\end{equation}
By the Sobolev embedding theorem on spheres, Theorem \ref{embedding on spheres},
for $f$ and its upper gradient $g\in L^Q(\Omega)$
we have for some $R\in (r, 2r)$
\[
\big(\underset{S_R(x)}{\osc}\,f\big)^p\le C\frac{r^{p}}{\mu(B_{2r}(x))}\int_{10\sigma B_{r}(x)}
g^p \, d\mu.
\]
Let $\rho$ be the zero extension of $g^p$ from $\Omega$ to the whole space.
Since $f$ is pseudomonotone, we have
\[
\begin{aligned}
\big(\underset{B_r(x)}{\osc}\,f\big)^p&\le \big(\underset{B_R(x)}{\osc}\,f\big)^p\\
&\le C_m^p\big(\underset{S_R(x)}{\osc}\,f\big)^p\\
&\le CC_m^p\frac{r^{p}}{\mu(B_{2r}(x))}\int_{10\sigma B_r(x)} g^p\, d\mu\\
&\le Cr^{p}\mathcal{M} \rho(z)
\end{aligned}
\]
for each $z\in B_r(x)$.
Integrating both sides with respect to $z$ over $B_r(x)$ and applying H\"older's inequality
as well as \eqref{lmb3}, we get
\[
\begin{aligned}
\big(\underset{B_r(x)}{\osc}\,f\big)^p& \le \frac{Cr^{p}}{\mu(B_r(x))} \int_{B_r(x)}\mathcal{M}\rho \, d\mu\\
&\le C\frac{R_0^p}{\mu(B_0)^{p/Q}}\Big(\int_{B_r(x)}(\mathcal{M}\rho)^{\frac{Q}{p}}\, d\mu\Big)^{\frac{p}{Q}}.
\end{aligned}
\]
The fact that $\frac{Q}{p}>1$ and the Hardy-Littlewood theorem imply that
$h=(\mathcal{M}\rho)^{\frac{Q}{p}}\in L^1(X)$, and now we have
\[
\big(\underset{B_r(x)}{\osc}\,f\big)^Q\le C\frac{R_0^Q}{\mu(B_0)}\|h \|_{L^{1}(B_r(x))}.
\]
Thus $f$ satisfies the condition \eqref{RR} with exponent $Q$ in the ball $B_0$, and thus locally in $\Omega$.
The local $Q$-absolute continuity follows, and then the absolute continuity in measure on $\Omega$ again
follows.
\end{proof}

\subsection{Absolute continuity of quasiconformal mappings}
Given a mapping $f\colon \Omega\to Y$, we define  for every $x\in \Omega$ and $r>0$
\[
L_f(x,r):=\sup\{d_Y(f(y),f(x))\colon d(y,x)\le r\},
\]
\[
l_f(x,r):=\inf\{d_Y(f(y),f(x))\colon d(y,x)\ge r\},
\]
and 
\[
H_f(x,r):=\frac{L_f(x,r)}{l_f(x,r)}.
\]
We interpret this to be $\infty$ if the denominator is zero.
We say that a homeomorphism $f\colon \Omega\to f(\Omega)\subset Y$ is quasiconformal if there is a number
$1\le H<\infty$ such that
\[
H_f(x):=\limsup_{r\to 0} H_f(x,r)\le H
\]
for all $x\in \Omega$. In Euclidean spaces, it is proved by Gehring \cite{Ge1,Ge2}
that for an exceptional set $E$ of
$\sigma$-finite $(n-1)$-dimensional Hausdorff measure and another
exceptional set $F$ of zero Lebesgue measure,
in the definition it suffices to assume $H_f(x)< \infty$ for all $x\in \Omega\setminus E$ and
$H_f(x)\le H<\infty$ for all $x\in \Omega\setminus F$. Heinonen and Koskela \cite{HK1}
show that $H_f(x)$ in the definition can be relaxed to 
\begin{equation}\label{hf}
h_f(x):=\liminf_{r\to 0} H_f(x,r)\le H.
\end{equation}
Moreover, Kallunki and Koskela \cite{KaKo} show that
similar exceptional sets are also admissible with the relaxed limit $h_f$.

The seminal work of Heinonen and Koskela \cite{HK2} develops the foundation of the theory of
quasiconformal mappings in metric spaces that satisfy certain bounds on their mass and geometry.
Mappings $f\colon X\to Y$ between two locally Ahlfors $Q$-regular metric measure spaces
supporting a $Q$-Poincar\'e inequality are carefully investigated. A homeomorphism
$f\colon  X\to Y$ is called quasiconformal if there is a constant $H<\infty$ such that $H_f(x)\le H$
for all $x\in X$ and $f$ is called quasisymmetric if $H_f(x,r)\le H$ for all $x\in X$ and $r>0$. 
Furthermore, quasisymmetric mappings are shown to be absolutely continuous in measure in
\cite[Corollary 7.13]{HK2}, \cite[Theorem 8.12]{HKST1}. 

 The equivalence of quasiconformal mappings and quasisymmetric mappings between metric spaces
 of locally $Q$-bounded geometry is obtained in \cite[Theorem 9.8]{HKST1}. The absolute continuity
 in measure of quasiconformal mappings in metric spaces with appropriate structure assumptions
 follows as a corollary.

Balogh, Koskela and Rogovin further show that if $X$ supports a $1$-Poincar\'e inequality,
and $h_f<\infty$ on $X\setminus E$
where $E$ is a set of $\sigma$-finite $\mathcal H^{Q-1}$-measure,
and $h_f\le H<\infty$ $\mu$-almost everywhere on $X$, then 
$f\in N^{1,Q}_\loc(X;Y)$ \cite[Theorem 5.1]{BKR}.
A recent result by the authors shows that the condition on $h_f$
can be relaxed by introducing an appropriate weight in the target space $Y$; we state a special
case of Corollary 1.2 of \cite{LaZh1} below.
By a weight we simply mean a nonnegative locally integrable function, and we consider the pointwise
representative
\begin{equation}\label{wY representative}
w_Y(y)=\liminf_{r\to 0}\frac{1}{\nu(B_r(y))}\int_{B_r(y)}w_Y\,d\nu,\quad y\in Y.
\end{equation}

\begin{thm}\label{previouspaper}
Let $(X,d,\mu)$ and $(Y,d_Y,\nu)$ be Ahlfors $Q$-regular spaces with $Q>1$,
and $X$ supports a $1$-Poincar\'e inequality.
In the space $Y$, let $w_Y>0$ be a weight represented by \eqref{wY representative}.
Let $\Omega\subset X$ be a bounded open set, and let $f\colon \Omega\to f(\Omega)$ be
a homeomorphism such that $f(\Omega)$ is open and $\int_{f(\Omega)}w_Y\,d\nu<\infty$.
Assume $E\subset \Omega$ is a set with $\sigma$-finite $(Q-1)$-Hausdorff measure such that $h_f(x)<\infty$
for $x\in \Omega\setminus E$ and $\frac{h_f(\cdot)^Q}{w(f(\cdot))}\in L^\infty(\Omega)$.
Then $f\in D^Q_{\loc}(\Omega;Y).$
\end{thm}

\begin{proof}[{Proof of Corollary \ref{Lusin}}]
	From Theorem \ref{previouspaper} we obtain
	$f\in D^{Q}(\Omega;Y)$.
Now the result follows from Theorem \ref{main}.
\end{proof}

\begin{remark}
Note that in Corollary \ref{Lusin} we are able, perhaps unexpectedly, to think of $Y$ as a
weighted space in order to show that $f\in AC^Q(\Omega;Y)$, even though $Q$-absolute continuity
is defined by essentially thinking of $Y$ as an unweighted space.

Another observation is that we obtain $Q$-absolute continuity, which is defined by considering
finite collections of arbitrary disjoint balls, from a condition on the quantity $h_f$ that only gives
pointwise, asymptotic control of $f$.
\end{remark}

A simple demonstration of applying the above corollary is given in the following example,
which uses a similar construction as \cite[Example 6.3]{LaZh1}.

\begin{example}\label{ex:plane}
	Consider the  square $\Omega:=(-1,1)\times (-1,1)$ on the unweighted plane
	$X=Y=\R^2$.
	Define the homeomorphism $f\colon \Omega\to f(\Omega)=(-1,1)\times (-2/3,2/3)$:
	\[
	f(x_1,x_2):=
	\begin{cases}
	(x_1,\frac 23 x_2^{3/2}),\quad x_2\ge 0,\\
	(x_1,-\frac 23 |x_2|^{3/2}),\quad x_2\le 0.
	\end{cases}
	\]
	We have $|Df|=\sqrt{1+|x_2|}$.
	Thus  $|Df|\in L^2(\Omega)$ and then clearly $f\in N^{1,2}(\Omega;f(\Omega))$.
	It is also easy to verify that $f$ is pseudomonotone,
	since $f$ is monotone in both $x_1$ and $x_2$.
	From the classical area formula, we can obtain directly that $f$ is in $AC^2(\Omega;f(\Omega))$.
	
	Overall, for such a regular mapping $f$, all of these properties are easy to verify directly,
	but it is also interesting to check what we can obtain from considering the quantity $h_f$.
	
	We easily find that $h_f(x_1,x_2)=|x_2|^{-1/2}$.
	Now $h_f\notin L^{\infty}(\Omega)$, and so previous results on quasiconformal mappings, such
	as those of \cite{BKR} or \cite{Wi}, do not give $f\in N^{1,2}(\Omega;f(\Omega))$, let alone
	$f\in AC^2(\Omega;f(\Omega))$.
	However, by considering the weight $w_Y(y_1,y_2)=|y_2|^{-2/3}$,
	we get $\frac{h_f(\cdot)^2}{w_Y(f(\cdot))}\in L^\infty(\Omega)$.
	We choose $E=(-1,1)\times \{0\}$.
	Now Corollary \ref{Lusin} gives $f\in AC^{2}(\Omega;f(\Omega))$.
\end{example}


\end{document}